\pdfoutput=1
\RequirePackage{ifpdf}
\ifpdf % We are running pdfTeX in pdf mode
\documentclass[pdftex]{sigma}
\else
\documentclass{sigma}
\fi

\numberwithin{equation}{section}

\newtheorem{Theorem}{Theorem}[section]
\newtheorem{Condition}[Theorem]{Condition}
\newtheorem*{ConditionN}{Condition~\ref{cond1}$'$}
\newtheorem{Lemma}[Theorem]{Lemma}
\newtheorem{Proposition}[Theorem]{Proposition}
 { \theoremstyle{definition}
\newtheorem{Definition}[Theorem]{Definition}
\newtheorem{Example}[Theorem]{Example}
\newtheorem{Remark}[Theorem]{Remark} }

\begin{document}

\allowdisplaybreaks

\newcommand{\arXivNumber}{1909.13211}

\renewcommand{\thefootnote}{}

\renewcommand{\PaperNumber}{009}

\FirstPageHeading

\ShortArticleName{New Examples of Irreducible Local Diffusion of Hyperbolic PDE's}

\ArticleName{New Examples of Irreducible Local Diffusion\\ of Hyperbolic PDE's\footnote{This paper is a~contribution to the Special Issue on Algebra, Topology, and Dynamics in Interaction in honor of Dmitry Fuchs. The full collection is available at \href{https://www.emis.de/journals/SIGMA/Fuchs.html}{https://www.emis.de/journals/SIGMA/Fuchs.html}}}

\Author{Victor A. VASSILIEV~$^{\dag\ddag}$}

\AuthorNameForHeading{V.A.~Vassiliev}

\Address{$^\dag$~Steklov Mathematical Institute of Russian Academy of Sciences, Moscow, Russia}
\EmailD{\href{mailto:vva@mi-ras.ru}{vva@mi-ras.ru}}
\URLaddressD{\url{http://www.mi-ras.ru/~vva/}}

\Address{$^\ddag$~National Research University Higher School of Economics, Moscow, Russia}

\ArticleDates{Received September 29, 2019, in final form February 18, 2020; Published online February 24, 2020}

\Abstract{Local diffusion of strictly hyperbolic higher-order PDE's with constant coefficients at all {\em simple} singularities of corresponding wavefronts can be explained and recognized by only two local geometrical features of these wavefronts. We radically disprove the obvious conjecture extending this fact to arbitrary singularities: namely, we present examples of diffusion at all non-simple singularity classes of generic wavefronts in odd-dimensional spaces, which are not reducible to diffusion at simple singular points.}


\Classification{35L67; 58K05; 32-04}

\renewcommand{\thefootnote}{\arabic{footnote}}
\setcounter{footnote}{0}

\section{Introduction}
This is a work in {\em lacuna theory} of hyperbolic PDE's initiated by I.G.~Petrovskii~\cite{Petrovskii45} and expanded significantly by J.~Leray~\cite{Leray}, M.~Atiyah, R.~Bott and L.~G{\aa}rding \cite{ABG70, ABG73, Gording77}, and others. However, the methods of the work belong mainly to singularity theory of differentiable functions.

There are two known sources of local diffusion (i.e., the local irregularity of continuations) of solutions of strictly hyperbolic partial differential equations with constant coefficients: non-singular points of wavefronts, at which the Davydova signature condition fails, and cuspidal edges at which the investigated local connected component of the complement of a wavefront is the ``bigger'' one. All cases of local diffusion of waves at {\em simple} (i.e., of classes $A_k$, $D_k$, $E_6$, $E_7$ or $E_8$) singular points of wavefronts of generic hyperbolic PDE's can be reduced to these two: the diffusion arises in such a component in a neighbourhood of a simple singularity if and only if the boundary of this component
contains points of one of these two basic types. In particular, this is true for all generic hyperbolic PDE's in the $\leq 7$-dimensional spaces.

We show that for non-simple singular points of wavefronts in ${\mathbb R}^N$, $N$ odd, this is not more the case. Namely, we indicate local components of complements of generic wavefronts at their points of
all ``parabolic'' singularity classes $P_8$, $X_9$, and $J_{10}$, such that solutions of the corresponding PDE's have diffusion at the most singular points of their boundaries only (i.e., at the points of these parabolic types), but are sharp at all simpler points.
The singularities of these three classes occur close to all other non-simple singularities of generic wavefronts, therefore such additional examples of diffusion also occur at all of them.

The proofs use a program counting topological types of morsifications of critical points of real analytic functions.

\subsection{Hyperbolic PDE's of higher orders (after \cite{ABG70, ABG73}, see also \cite{GP,APLT})}

Let $F$ be a linear partial differential operator with constant coefficients: $F$ has the form of a~polynomial (generally with complex coefficients) in variables $\partial/\partial x_j$ where $x_j$ are the coordinates in~${\mathbb R}^N$. It is convenient to identify these variables $\partial/\partial x_j$ with coordinates~$\xi_j$ in dual space~$\check {\mathbb R}^N$ of our space ${\mathbb R}^N$, with pairing operation \begin{gather*}
\langle (x_1, \dots, x_N),(\xi_1, \dots, \xi_N)\rangle =\sum_{j=1}^N x_j \xi_j.
\end{gather*}
The projectivization of this dual space is denoted by $\check {\mathbb R}P^{N-1}$; its $k$-dimensional subspaces correspond via the {\em projective duality} to $(N-k-2)$-dimensional subspaces in ${\mathbb R}P^{N-1}$.

Let $d$ be the degree of this polynomial $F$.
The following {\em Cauchy problems} in the half-space ${\mathbb R}^N_+= \{x\,|\,x_1>0\}$ for this operator are considered. Given any regular function $\varphi$ in this half-space, and a collection of $d$ regular functions $\psi_k$, $k=0, 1, \dots, d-1,$ defined on its boundary hyperplane ${\mathbb R}^{N-1}_1 \equiv \{x_1 =0\}$ of this half-space, the task is to find a function $u$ in ${\mathbb R}^N_+$, such that $F(u) \equiv \varphi$ in ${\mathbb R}^N_+$, and $\frac{\partial^k}{\partial x_1^k} u \equiv \psi_k$ on hyperplane ${\mathbb R}^{N-1}_1$ for any $k=0, \dots, d-1$.

A {\em fundamental solution} of our operator $F$ is any distribution in ${\mathbb R}^N$ such that this operator applied to it is equal to the delta function at $0$. Operator $F$ is called {\em hyperbolic} (or sometimes {\em hyperbolic in the sense of Petrovskii}) if it admits a fundamental solution whose support belongs to a proper cone $S(F)$ in the half-space ${\mathbb R}^N_+$ with the vertex $0$ (i.e., this vertex is the unique intersection point of this cone with ${\mathbb R}^{N-1}_1$).
Such a fundamental solution (if it exists) is uniquely determined by $F$; it is called {\em principal fundamental solution} of $F$ and is denoted by $E(F)$. It may be considered as a (generalized) wave caused by an instant pointwise perturbation and propagating in the future ($x_1>0$) part of our space-time only.

If $F$ is hyperbolic, then any Cauchy problem of the described form has a unique solution, which is regular and depends continuously on the data $\{\varphi; \psi_0, \dots, \psi_{d-1}\}$ of the problem; moreover, the value of this solution at any point $x \in {\mathbb R}^N_+$ depends only on the values of these data in the capsized cone $x-S(F)$ with the vertex at point $x$. Indeed, such a solution is provided by the convolution of $E(F)$ with the data of our problem.

The hyperbolicity property implies strong conditions on the {\em principal symbol} of $F$ (i.e.,
the homogeneous part $F_d$ of highest degree of polynomial $F$). First of all, if the operator $F$ is hyperbolic then this principal part is necessarily a real (up to a constant factor) polynomial; therefore studying solutions of the corresponding equation we can and will assume that it is just real. Further, the cone $A(F) \subset \check {\mathbb R}^N$ of zeros of this principal part $F_d$ should have exactly $d$ real intersection points (maybe counted with multiplicities) with any line in $\check {\mathbb R}^N$ parallel to the $\xi_1$ axis (i.e., to the line orthogonal to the hyperplane ${\mathbb R}^{N-1}_1$).

If this cone $A(F)$ is a non-singular manifold outside the origin in $\check {\mathbb R}^N$, then this condition is also sufficient for hyperbolicity of entire operator $F$. Otherwise some additional conditions on the lower terms of $F$ should be satisfied.

Hyperbolic operators with such smooth cones $A(F)$ are called {\em strictly hyperbolic}. They form a dense subset in the space of all hyperbolic operators. We consider in this work only strictly hyperbolic operators.

The principal fundamental solution of any hyperbolic operator is regular (i.e., locally coincides with appropriate non-singular analytic functions) everywhere outside the {\em wavefront} $W(F)$, which is a conic (i.e., invariant under positive dilations) closed semi-algebraic subvariety of positive codimension in ${\mathbb R}^N_+$; moreover the support of this principal fundamental solution lies in the convex hull of the wavefront.

Namely, the wavefront of a {\em strictly} hyperbolic operator consists of points $x \in {\mathbb R}^N_+$ such that their orthogonal hyperplanes in the space $\check {\mathbb R}^N$ are tangent to the cone $A(F)$. For general hyperbolic operators also some points corresponding to the hyperplanes ``not in general position'' with the singular set of $A(F)$ should be added.

Given a point $x \in W(F)$, the principal fundamental solution $E(F)$ is called {\em holomorphically sharp} in a local (close to $x$) connected component $C$ of the regularity domain ${\mathbb R}^N \setminus W(F)$ if there is a smooth analytic function in a neighbourhood of $x$ in ${\mathbb R}^N$ coinciding with $E(F)$ in the intersection of this neighbourhood and this component $C$. $E(F)$ is {\em $C^\infty$-sharp} in such a~component if it can be extended to a $C^\infty$-smooth function on the closure of this component.

\begin{Example}
The wavefront of the standard wave operator
\begin{gather*}
\frac{\partial^2}{\partial x_1^2} - c^2\left(\frac{\partial^2}{\partial x_2^2} +\frac{\partial^2}{\partial x_3^2} \right)
\end{gather*}
 in ${\mathbb R}^3$ is the cone
\begin{gather*}c^2x_1^2= x_2^2+x_3^2, \qquad x_1 \geq 0;
\end{gather*}
 the principal fundamental solution of this operator is sharp at all its points in the ``exterior'' component of the complement of $W(F)$, and is not in the interior one (growing there asymptotically as $\text{(distance from the wavefront)}^{-1/2}$ close to the regular points of this cone). On the contrary, principal fundamental solutions of wave operators in space-time of even dimension greater than~2 are sharp from both sides of the cone.
\end{Example}

A local component of ${\mathbb R}^N \setminus W(F)$, in which the fundamental solution is sharp, is called a~(holomorphic or $C^\infty$-) {\em local lacuna}. The negation of sharpness is called a {\em diffusion} of waves.

 \begin{Remark}[on terminology] Roughly, elementary wave~$E(F)$ is sharp at a point of its front, if its singularity behaves like a squall, not predictable by the behaviour of the regular part of the wave before meeting with it. Diffusion is an opposite situation: in this case the shock front spreads some signs of its occurrence around it. The word ``lacuna'' was originally applied by Petrovskii to the domains where the fundamental solution is just equal to zero; so it indicated also the areas in~${\mathbb R}^N_+$ and~${\mathbb R}^{N-1}_1$ such that the values of the data of Cauchy problem in them are not important for the behavior of the solution of the problem in a given point~$x$. The {\em local} lacunas were introduced in \cite{ABG70, ABG73}.
\end{Remark}

Wavefronts can be singular along the rays in ${\mathbb R}^N_+$, which correspond by projective duality to {\it parabolic points} of the projectivization $A^*(F) \subset \check {\mathbb R}P^{N-1}$ of the cone $A(F) \subset \check {\mathbb R}^N$, i.e., the points at which the second fundamental form of this hypersurface in $\check {\mathbb R}P^{N-1}$ is degenerate.

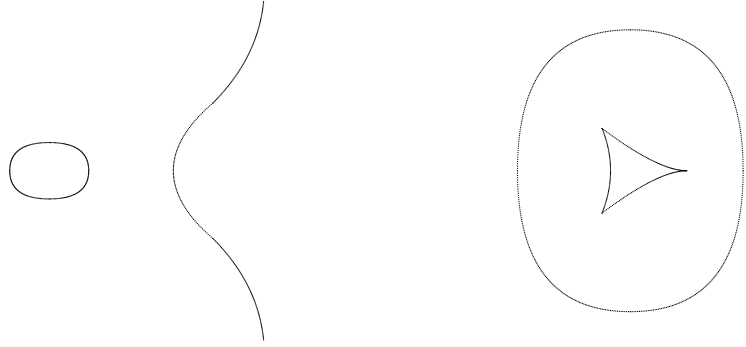
\begin{figure}\centering
\unitlength 0.75mm
\linethickness{0.4pt}
\begin{picture}(130,62)
\bezier{60}(0,30)(0,35)(7,35)
\bezier{60}(7,35)(14,35)(14,30)
\bezier{60}(14,30)(14,25)(7,25)
\bezier{60}(7,25)(0,25)(0,30)
\bezier{100}(45,0)(44,10)(36,18)
\bezier{100}(36,18)(22,30)(36,42)
\bezier{100}(45,60)(44,50)(36,42)
\bezier{75}(105,22.5)(108,30)(105,37.5)
\bezier{75}(105,37.5)(115,30)(120,30)
\bezier{75}(120,30)(115,30)(105,22.5)
\bezier{120}(110,5)(130,5)(130,30)
\bezier{120}(110,5)(90,5)(90,30)
\bezier{120}(110,55)(130,55)(130,30)
\bezier{120}(110,55)(90,55)(90,30)
\end{picture}
\caption{Projectivized zero set and wavefront of a cubic operator in ${\mathbb R}^3$.}\label{cubic}
\end{figure}

\begin{Example} The projectivization of the zero set of a non-singular homogeneous polynomial in $\check {\mathbb R}^3$ is a cubic curve in the projective plane $\check {\mathbb R}P^2$, which can consist either of two components as in Fig.~\ref{cubic} (left) or only the right-hand one of these components. The one-component case never is hyperbolic, and the two-component one is hyperbolic if the $\xi_1$-axis is directed inside the domain bounded by its left-hand component. The wavefront of the latter operator looks as a~cone in~${\mathbb R}^3_+$ whose image in ${\mathbb R}P^2$ is shown in Fig.~\ref{cubic} (right): its three singular points come from three inflection points of the cubic curve (one of which is in our picture placed at infinity).
\end{Example}

\begin{figure}\centering
\unitlength 1.00mm
\linethickness{0.4pt}
\begin{picture}(129.00,58.00)
\bezier{96}(0.00,26.00)(0.33,16.00)(10.00,6.00)
\bezier{96}(35.67,27.67)(36.00,17.67)(45.67,7.67)
\bezier{16}(35.33,24.00)(35.00,21.33)(34.33,19.67)
\bezier{16}(32.67,16.33)(32.00,14.67)(30.67,13.33)
\bezier{56}(-4.00,13.00)(-1.00,17.00)(00.00,26.00)
\bezier{108}(45.67,7.67)(29.67,10.67)(19.33,10.67)
\bezier{56}(19.33,10.67)(10.33,10.67)(10.00,6.00)
\bezier{108}(35.67,27.67)(19.67,30.67)(9.33,30.67)
\bezier{56}(9.33,30.67)(0.33,30.67)(0.00,26.00)
\bezier{24}(30.67,13.33)(28.00,14.00)(25.00,14.50)
\bezier{24}(22.33,14.85)(20.00,15.33)(16.33,15.67)
\bezier{20}(13.33,16.00)(11.67,16.33)(8.00,16.33)
\bezier{28}(3.33,16.33)(-0.33,16.33)(-3.33,14.67)
\bezier{72}(113.00,25.00)(104.00,23.00)(95.00,25.00)
\bezier{100}(95.00,25.00)(105.00,32.00)(109.80,42.55)
\put(110.10,43.00){\circle{1.00}}
\bezier{76}(110.00,42.55)(109.00,30.00)(113.00,25.00)
\bezier{152}(113.00,25.00)(94.00,18.00)(83.00,4.00)
\bezier{140}(95.00,25.00)(113.00,16.00)(120.00,3.00)
\bezier{208}(120.00,3.00)(122.00,34.00)(129.00,54.00)
\bezier{204}(92.00,54.00)(85.00,24.00)(83.00,4.00)
\bezier{152}(92.00,54.00)(116.00,59.00)(129.00,54.00)
\bezier{12}(102.75,20.45)(103.33,22.00)(103.67,23.33)
\bezier{24}(104.33,25.33)(105.33,28.33)(106.10,31.00)
\bezier{16}(106.67,32.67)(107.00,33.33)(108.00,36.67)
\bezier{12}(108.67,38.67)(109.33,40.67)(109.67,41.67)
\put(94.00,28.00){\makebox(0,0)[cc]{$A_2$}}
\put(115.00,28.00){\makebox(0,0)[cc]{$A_2$}}
\put(113.00,45.00){\makebox(0,0)[cc]{$A_3$}}
\put(10.00,33.50){\makebox(0,0)[cc]{$A_2$}}
\put(101.00,11.00){\makebox(0,0)[cc]{$3$}}
\put(87.00,49.00){\makebox(0,0)[cc]{$1$}}
\put(77,22){$2$}
\put(80,23){\vector(1,0){22}}
\end{picture}
\caption{Cuspidal edge and swallowtail.}\label{swallowtail}
\end{figure}
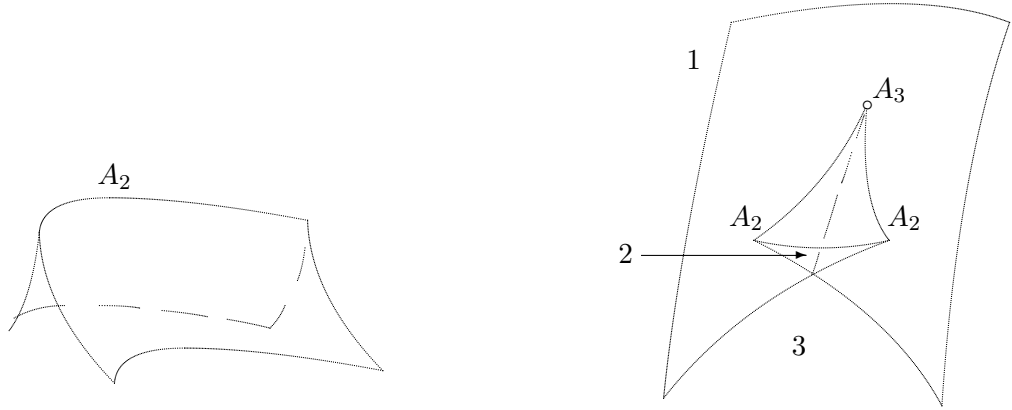

The properties of sharpness and diffusion at the points of a wavefront are invariant under dilations of ${\mathbb R}^N$. Therefore we will consider projectivizations $W^*(F) \subset {\mathbb R}P^{N-1}$ of wavefronts and speak on the diffusion or sharpness in a local connected component of ${\mathbb R}P^{N-1} \setminus W^*(F)$ if it holds in the preimage of this component under the obvious projection ${\mathbb R}^N_+ \setminus W(F) \to {\mathbb R}P^{N-1} \setminus W^*(F)$. These properties are related with the local geometry of the projectivized wavefront $W^*(F)$, which by construction is the variety projective dual to the projectivization $A^*(F) \subset \check {\mathbb R}P^{N-1}$ of the cone~$A(F)$. Two basic types of points of wavefronts spreading the diffusion in such components are described in the following two subsections.

A classification of singular points of projectivized wavefronts of strictly hyperbolic operators (and, more generally, of hypersurfaces projective dual to smooth ones) coincides with a classification of parabolic points of smooth hypersurfaces, and hence with a classification of critical points of smooth functions, see, e.g.,~\cite{AVG82}. The relation between them is provided by {\em generating functions} of these fronts, see Section~\ref{go} below. In particular, non-singular points of wavefronts correspond to non-parabolic points of $A^*(F)$ and to Morse critical points of generating functions, while the simplest standard singularities, {\em semicubic cuspidal edges} and {\em swallowtails}, shown correspondingly in the left and the right parts of Fig.~\ref{swallowtail}, are related with critical points of types~$A_2$ and~$A_3$. For the surfaces in ${\mathbb R}P^{N-1}$ with arbitrary $N\geq 4$, these pictures should be multiplied by a piece of ${\mathbb R}^{N-4}$. The singularities shown in Fig.~\ref{cubic} (right) for case $N=3$ have type~$A_2$: the cubic inflection points of curve $A^*(F)$ with local equation $\xi_0 = \xi_1^3 + O\big(\xi_1^4\big)$ in $\check {\mathbb R}P^2$ result in cusps of projective dual curve, which are locally ambient diffeomorphic to the semicubic parabola $\big\{x^3 + y^2 =0\big\}$; they appear as transversal slices of the picture of Fig.~\ref{swallowtail}.

\subsection{Davydova condition at non-singular points of a wavefront} \label{db}

A non-singular piece of a projectivized wavefront locally divides ${\mathbb R}P^{N-1}$ into two parts. Choose a normal direction to this wavefront $W^*(F)$ at a point of such a piece and consider the second fundamental form of the wavefront with respect to this direction. This quadratic form in $N-2$ variables is non-degenerate as the operator is strictly hyperbolic.

\begin{Definition}The {\em Davydova condition} \cite{Dav} is satisfied if the positive inertia index of this quadratic form is even.
\end{Definition}

If this condition is not satisfied, then we surely have local diffusion in the component of the complement of the wavefront, indicated by the chosen normal direction (see \cite{Dav}). Moreover, the analytic continuation of the fundamental solution from this component to a neighbourhood of such a point of the wavefront in ${\mathbb C}^N$ has a two-fold ramification along the (complexified) wavefront. Conversely, if the Davydova condition holds, then the wavefront is sharp in this local component: this was proved in \cite{Bor} by analytic estimates, and explained in \cite{ABG73} in terms of monodromy and removable singularities.

\subsection{Cuspidal edges and generating functions} \label{go}

The simplest singularities of a wavefront in ${\mathbb R}P^{N-1}$ are its {\em cuspidal edges}, at whose points it is locally ambient diffeomorphic to the product of the semicubical parabola and ${\mathbb R}^{N-3}$, see Figs. \ref{cubic} (right) and \ref{swallowtail} (left). Space ${\mathbb R}P^{N-1}$ is locally divided by such a hypersurface into two parts, one of which is ``bigger'', and the other one ``smaller''. According to \cite{Gording77}, the principal fundamental solution is never sharp in the bigger component, but can be sharp in the smaller one in some additional conditions. To describe these conditions, we need the following notion of generating functions of wavefronts.

Let $X\in {\mathbb R}P^{N-1}$ be a point of a projectivized wavefront. The corresponding hyperplane in~$\check {\mathbb R}P^{N-1}$ (i.e., projectivization of the hyperplane in~$\check {\mathbb R}^N$ orthogonal to the line $\{X\} \subset {\mathbb R}^N$) is tangent to projectivized zero set $A^*(F) $ of the principal symbol of $F$ at a point $\xi$. Choose affine co\-or\-di\-nates $(\xi_0, \dots, \xi_{N-2})$ in $\check {\mathbb R}P^{N-1}$ with the origin in this point in such a way that this tangent plane is defined by equation $\xi_0=0$. Hypersurface~$A^*(F)$ can be then specified close to~$\xi$ by an equation of the form
\begin{gather}\label{genfun}
\xi_0=f(\xi_1, \dots, \xi_{N-2}),
\end{gather}
where $f(0)=0$, ${\rm d}f(0)=0.$ Variables $\xi_1, \dots, \xi_{N-2}$ form a system of local coordinates on $A^*(F)$. The projective duality map $A^*(F) \to {\mathbb R}P^{N-1}$, and also the germ of its image at~$X$, are then determined by our function~$f$: for any point of the variety $A^*(F)$ we draw the tangent hyperplane at this point, and mark the corresponding point in~${\mathbb R}P^{N-1}$.

\begin{Definition}\label{genf}The function germ~$f$ defined in this way is called a {\em projective generating function} of the wavefront at point $X \in {\mathbb R}P^{N-1}$ corresponding to the hyperplane $\{\xi_0=0\} \subset \check {\mathbb R}P^{N-1}$.
\end{Definition}

It follows easily from this definition that any two generating functions of the same germ of the wavefront, defined by different choices of affine coordinates $\xi_i$, can be obtained one from the other by a diffeomorphism of space of arguments and a dilation of the target line by a non-zero (maybe negative) coefficient.

\begin{Remark} If occasionally the hyperplane corresponding to $X$ is tangent to hypersur\-fa\-ce~$A^*(F)$ at several its points, then such a map is defined in a neighbourhood of any of these points; the germ of the wavefront at point~$X$ will then consist of several locally irreducible branches corresponding to all these tangency points and described by their projective generating functions.
\end{Remark}

If generating function $f$ is Morse, then the corresponding piece of the wavefront is smooth (and was considered in the previous subsection). A cuspidal edge of the wavefront occurs when~$f$ has the simplest non-Morse singularity of type~$A_2$, see, e.g., \cite{ALGV, AVG82}. According to L.~G{\aa}rding~\cite{Gording77}, a wavefront is sharp at its cuspidal edge if and only if dimension~$N$ is odd, the inertia indices of the quadratic part of the function germ $f$ are even, and the investigated component of the complement of the wavefront is the ``smaller'' one. (The rank of the quadratic part of a function of this class depending on $N-2$ variables is equal to $N-3$, hence in the case of odd $N$ both inertia indices are of the same parity.) For example, these conditions are satisfied for the interior component in Fig.~\ref{cubic} (right).

\begin{Remark}\looseness=1 The part ``only if'' of this G{\aa}rding's result claims that diffusion at cuspidal edges appears in some seven cases, depending on the parities of~$N$ and inertia indices of the quadratic part of~$f$, and on the choice of a component. In some six out of these seven cases, diffusion follows already from the Davydova's obstruction: we can indicate a smooth piece of wavefront approaching the cuspidal edge, such that sharpness in our component fails at the points of this piece. The only exceptional case is that of odd~$N$ and odd inertia indices of~$f$: in this case the wavefront is sharp for the ``bigger'' component at all neighboring smooth points of its surface, however regularity of the fundamental solution fails when we approach its singular stratum.
\end{Remark}

\subsection{Other simple singularities}

The most common equivalence relation of singularity classes of germs of functions $\big({\mathbb R}^n,0\big) \to {\mathbb R}$ is provided by the group $\operatorname{Diff}_0$ of local diffeomorphisms $\big({\mathbb R}^n, 0\big) \to \big({\mathbb R}^n,0\big)$ acting on functions by composition with these diffeomorphisms. The space of function germs splits into the orbits of this action. This splitting is not discrete: some singularity classes can depend on continuous parameters separating continuously many orbits of this action.

\looseness=1 A natural primary segment of the classification of singularities of smooth functions (and hence also of wavefronts) is formed by so-called {\em simple singularities} $A_k$, $D_k$, $E_6$, $E_7$ or~$E_8$, see~\cite{AVG82}. For this part of the classification the splitting into the $\operatorname{Diff}_0$-orbits is still discrete: by definition these are the singularities such that a neighbourhood of any of them in the function space is covered by representatives of only finitely many orbits of this action. The name ``simple'' and canonical notation of these classes come from a deep and multiform relation of these classes with the simple Lie algebras. All local lacunas close to singular points of wavefronts of these types were counted in~\cite{Vassiliev86} (except for cases~$A_2$ and~$A_3$ of cuspidal edges and swallowtails, which were in detail described by L.~G{\aa}rding~\cite{Gording77}). In~\cite{Vassiliev92} an easy geometric characterization of these lacunas was found: it turned out that all cases of the diffusion close to any simple singular points of generic wavefronts can be reduced to the two cases described above. Namely, the following theorem was proved there (the technical notion of projective versality used in it and satisfied for generic wavefronts will be defined in Section~\ref{prove}).

\begin{Theorem}[see \cite{Vassiliev92}] \label{v92}
If projectivized wavefront $W^*(F)$ of a strictly hyperbolic operator $F$ has a simple singularity at its point $X$ and is {\em projective versal} at this point, then the principal fundamental solution of $F$ is sharp at point $X$ in a certain local connected component of the complement of $W^*(F)$ if and only if
\begin{enumerate}\itemsep=0pt
\item[$1)$] the Davydova condition is satisfied at all smooth point of $W^*(F)$ in the boundary of this component, and
\item[$2)$]this boundary does not contain the cuspidal edges of $W^*(F)$, close to which our component is the ``bigger'' one.
\end{enumerate}
\end{Theorem}

\begin{Example}[see~\cite{Gording77}] Close to a swallowtail (see Fig.~\ref{swallowtail}), only the following components of the complement of the wavefront are local lacunas: component~2 if $N$ is odd and the positive inertia index $i_+$ of the quadratic part of the generating function is even; component~3 if $i_+$ is odd and $N$ is arbitrary.
\end{Example}

\begin{Remark}\quad
\begin{enumerate}\itemsep=0pt
\item If $N$ is even, then condition 2) in this theorem is unnecessary: if it is not satisfied at an edge point, then condition 1) also fails at some neighboring smooth points of the wavefront.
\item The restrictions from Sections~\ref{db} and~\ref{go} prove the part ``only if'' of this theorem (which holds also for not necessarily projective versal wavefronts); a proof of the part ``if''
is based on the properness of the so-called Lyashko--Looijenga maps of miniversal deformations of simple singularities, see \cite{Loo,Loor}.
\end{enumerate}
\end{Remark}

A natural question arising from this theorem (and formulated explicitly by V.P.~Pa\-la\-modov about 1991, see~\cite{Vassiliev92}) asks whether it can be extended to all singularity classes of wavefronts of strictly hyperbolic operators. We show below that the answer is negative in the case of odd~$N$ for {\em all} non-simple singularity classes.

\subsection{Generating families of wavefronts and versality}\label{prove}

The projectivized wavefront $W^*(F) \subset {\mathbb R}P^{N-1}$ can be considered close to any its point $X$ as the real discriminant variety of a deformation of the corresponding generating function~(\ref{genfun}). Namely, this deformation depends on the $(N-1)$-dimensional parameter ${\bf x} =(x_0, x_1, \dots, x_{N-2})$ and consists of functions $f_{\bf x}$ defined by formula
\begin{gather}
f_{\bf x} \equiv f(\xi_1, \dots, \xi_{N-2}) - x_0 - \sum_{j=1}^{N-2} x_j \xi_j . \label{prover}
\end{gather}
The parameters $x_0, \dots, x_{N-2}$ of this deformation form a coordinate system in an affine chart of space ${\mathbb R}P^{N-1}$ with the origin at the point $X$: this chart consists of hyperplanes in $\check {\mathbb R}P^{N-1}$ defined by equation
\begin{gather}\label{exex}
\xi_0 =x_0+ \sum_{j=1}^{N-2} x_j \xi_j
\end{gather}
in local coordinates $\xi_i$. Such a point ${\bf x}\in {\mathbb R}P^{N-1}$ belongs to the {\em real discriminant} of deformation~(\ref{prover}) (i.e., $f_{\bf x}$ has a real critical point close to $0$ with critical value $0$) if and only if the hyperplane in $\check {\mathbb R} P^{N-1}$ given by~(\ref{exex}) is tangent to~$A^*(F)$; so the real discriminant of (\ref{prover}) is exactly the variety projective dual to~$A^*(F)$.

We will consider the case when this deformation is sufficiently representative, namely is a~{\em versal deformation} of~$f$. Let us remind this notion (for an expanded description see~\cite{AVG82}).

Suppose we have a smooth function $\varphi(y)\colon {\mathbb R}^n \to {\mathbb R}$ with $d\varphi(0)=0$, and its deformation $\Phi(y, \lambda) \colon \big({\mathbb R}^n \times {\mathbb R}^l, 0\big) \to ({\mathbb R},0)$, i.e., a family of functions $\varphi_\lambda\colon {\mathbb R}^n \to {\mathbb R}$ depending on the parameter $\lambda \in {\mathbb R}^l$, $\varphi_0 \equiv \varphi$. This deformation is called {\em versal} if any other deformation $\Psi\colon \big({\mathbb R}^n \times {\mathbb R}^k ,0\big) \to ({\mathbb R},0) $ of the same function $\varphi$ can be reduced to it by an appropriate map of parameters and a family of local diffeomorphisms of the space ${\mathbb R}^n$ depending on these parameters: in formulas, there should be a smooth map $\theta\colon \big({\mathbb R}^k,0\big) \to \big({\mathbb R}^l,0\big)$ and a family of local (i.e., defined in a neighbourhood of the origin) diffeomorphisms $H_{\varkappa}\colon {\mathbb R}^n \to {\mathbb R}^n$ smoothly depending on the parameter $\varkappa \in {\mathbb R}^k$, $H_0 \equiv \mbox{\{the identity map\}}$, such that
\begin{gather}\label{versdef}
\Psi(y,\varkappa) = \Phi(H_{\varkappa}(y),\theta(\varkappa)) \ \mbox{ for any $y$ and $\varkappa$ sufficiently close to the origin} .
\end{gather}

It is known that
\begin{itemize}\itemsep=0pt
\item For any function singularity $\varphi$ with finite Milnor number $\mu(\varphi)$, almost all its deformations depending on $\geq \mu(\varphi)$ parame\-ters are versal;
\item such a singularity does not have versal deformations depending on less than $\mu(\varphi)$ parame\-ters (versal deformations depending on exactly $\mu(\varphi)$ parameters are called {\em miniversal}),
\item if $\Phi$ is a versal deformation of $\varphi$ then for any sufficiently small perturbation $\tilde \varphi$ of our singularity $\varphi$ there is a point $\lambda \in {\mathbb R}^l$ close to the origin such that the functions $\tilde \varphi$ and $\varphi_\lambda$ can be transformed one into the other by a local diffeomorphism in~${\mathbb R}^n$;
\item if $\Phi$ is a miniversal deformation of $\varphi$ (thus depending on $\mu(\varphi)$ parameters), and $\Psi$ an arbitrary versal deformation of~$\varphi$ depending on~$k$ parameters, then the corresponding map~$\theta$ in~(\ref{versdef}) is a~submersion close to the origin in ${\mathbb R}^k$, sending the discriminant of deformation~$\Psi$ to the discriminant of deformation $\Phi$.
\end{itemize}

Similar notion and facts hold for singularities of holomorphic functions of complex variables.

\begin{Definition}Deformation (\ref{prover}) is called a (projective) generating family of the wave\-front~$W^*(F)$ at point~$X$. The hypersurface $A(F)$ is {\em projective versal} at its point $X$ if the corresponding deformation~(\ref{prover}) is versal, see~\cite{AVG82}.
\end{Definition}

\begin{Example} Suppose that generating function $f$ has a singularity of type~$A_k$ at a point $\Xi \in \check{\mathbb R} P^{N-1}$. Then there is a smooth parametric curve $u\colon \big({\mathbb R}^1,0\big) \to \big({\mathbb R}^{N-2},0\big)$ such that the function $|{\rm grad}\, f|$ grows as the $k$th degree of the parameter along it. The linear hull of the first~$j$ derivatives of this curve at~$0$ is uniquely defined by this condition for any $j \leq k-1$. The hypersurface with equation $\xi_0 = f(\xi_1, \dots, \xi_{N-2})$ is projective versal at point~$\Xi$ if and only if all these first $k-1$ derivatives are linearly independent in~${\mathbb R}^{N-2}$.
\end{Example}

The projective versality at simple singular points is a condition of general position, however for more complicated singularity classes, which split into continuous families of orbits of the group $\operatorname{Diff}_0$ of local changes of variables $\big({\mathbb R}^n,0\big) \to \big({\mathbb R}^n,0\big)$, this is generally not true. For instance, a generic hyperbolic operator of order $\geq 3$ in~${\mathbb R}^8$ can have singularities of type~$P_8$ (see~\cite{AVG82} and the next Section~\ref{parara}) in some discrete set of points of the projectivized wavefront; the generating family~(\ref{prover}) is in this case transversal to the entire stratum~$\{P_8\}$ in the space of function singularities, but not to its particular orbits, and hence is not versal.

\subsection{Parabolic singularities and main result}\label{parara}

Parabolic singularities of function germs (see \cite{AVG82}) form the next natural family of singularity classes after the simple ones. Besides some nice algebraic properties distinguishing them, it is important for us that they are {\em confining} for the family of simple singularities (see \cite{AVG82}), that is, any non-simple singularity can be turned to a parabolic one by an arbitrarily small perturbation. Therefore, proving that some property (e.g., diffusion) holds close to any non-simple singularity, it is enough to prove it for these classes only.

The parabolic classes are listed in Table \ref{t4}. Any function germ ${\mathbb R}^n \to {\mathbb R}$ of this class can be reduced by a smooth change of variables to one of the normal forms indicated in this table. The letter $Q$ in it denotes non-degenerate quadratic forms in the variables missing in the first parts of these formulas, e.g.,~$\pm \xi_4^2 \pm \dots \pm \xi_n^2$ for singularities $P_8^1$ and $P_8^2$, and $\pm \xi_3^2 \pm \dots \pm \xi_n^2$ for five other singularity types. So, the {\em corank} of a parabolic singularity (i.e., the number of essential variables) is equal to 3 in the case of $P_8$ singularities, and to 2 for $X_9$ and $J_{10}$.

\begin{table}[t]\centering
\caption{Real parabolic singularities}\label{t4}\vspace{1mm}
\begin{tabular}{|c|l|l|}
\hline
Notation & \qquad \qquad Normal form & Restrictions \\
\hline
$P_8^1$ &
$\xi_1^3 + \alpha \xi_1 \xi_2 \xi_3 + \xi_2^2\xi_3 + \xi_2\xi_3^2+ Q$ &
$\alpha > -3 $\tsep{2pt} \\ $P_8^2$ &
$\xi_1^3 + \alpha \xi_1 \xi_2 \xi_3 + \xi_2^2\xi_3 + \xi_2\xi_3^2+ Q$ &
$\alpha < -3 $ \tsep{2pt} \\
$\pm X_9$ & $\pm \big(\xi_1^4 + \alpha \xi_1^2 \xi_2^2 + \xi_2^4 + Q\big)$ & $\alpha > -2$ \tsep{2pt} \\
$X_9^1$ & $ \xi_1\xi_2\big(\xi_1^2 + \alpha \xi_1\xi_2 + \xi_2^2\big) + Q$ &
$\alpha^2 < 4 $ \tsep{2pt} \\
$X_9^2$ & $ \xi_1\xi_2(\xi_1 + \xi_2)(\xi_1 + \alpha \xi_2) + Q$ &
$\alpha \in (0,1) $ \tsep{2pt} \\
$J_{10}^3$ & $\xi_1\big(\xi_1 - \xi_2^2\big)\big(\xi_1 - \alpha \xi_2^2\big) + Q$ &
$\alpha \in (0,1) $ \tsep{2pt} \\
$J_{10}^1$ & $\xi_1\big(\xi_1^2 + \alpha \xi_1\xi_2^2 + \xi_2^4\big) + Q$ &
$\alpha^2 < 4 $ \bsep{2pt}\\
\hline
\end{tabular} \end{table}

\begin{Proposition}[see \cite{AVG82}] \label{prim}
All sufficiently small perturbations of a function germ with a~para\-bo\-lic critical point have only simple critical points or parabolic critical points of the same class as the perturbed one.
\end{Proposition}

Indeed, a small perturbation cannot increase the Milnor number, or the corank or the multiplicity of a singularity.

\begin{Theorem} \label{maint} Suppose that $N$ is odd, the projectivized wavefront $W^*(F)\subset {\mathbb R}P^{N-1}$ of a~strict\-ly hyperbolic operator $F$ has a singularity of one of classes indicated in Table~{\rm \ref{t4}} at a point $X \in W^*(F)$, and one of two conditions is satisfied:
\begin{enumerate}\itemsep=0pt
\item[$a)$] the positive inertia index of the quadratic part $Q$ of the corresponding generating function is odd;
\item[$b)$] the singularity type of $W^*(F)$ at point $X$ is one of the following four: $X_9^1$, $X_9^2$, $J_{10}^3$, or~$J_{10}^1$.
\end{enumerate}
Then
\begin{enumerate}\itemsep=0pt
\item[$1)$] the principal fundamental solution $E(F)$ has a holomorphic diffusion in {\em all} local components of the complement of the wavefront $W^*(F)$ at the point~$X$;
\item[$2)$] if the singularity of the wavefront at point $X$ is projective versal, then
there is a local $($close to the point~$X)$ connected component of the complement of the wavefront, such that the fundamental solution of~$F$ is sharp in this component at all points of the wavefront in the boundary of this component, except for the points of the most singular stratum $($i.e., for parabolic singular points of the corresponding class$)$.
\end{enumerate}
\end{Theorem}

Both statements of this theorem in Case~(b) follow from the same statements of Case~(a) for the same singularity types, because the four classes mentioned there are invariant under the multiplication of functions by~$-1$, which preserves also the condition of sharpness, but changes (in the case of odd~$N$) the parities of inertia indices of quadratic forms of rank~$N-4$. Therefore we will consider in our proofs only Case~(a), i.e., assume that positive inertia indicex of quadratic form~$Q$ is odd. For a proof of Statement~1 of this theorem see~\cite{APLT} for all parabolic singularity classes except for~$P_8^2$, and~\cite{V16} for~$P_8^2$; see also Remark~\ref{rem3} in p.~\pageref{rem3} below.
Statement~2 in Case~(a) will be proved in Section~\ref{p8} for singularities~$P_8^1$ and~$P_8^2$, and in Section~\ref{cor2} for all the other parabolic singularity types.

\begin{Remark} The singular points of wavefronts of three parabolic classes not mentioned in Case~(b) of this theorem in the case of odd~$N$ and even positive inertia index of form~$Q$ do have local lacunas in their neighborhoods, see~\cite{V16}.
\end{Remark}

\section{Topological reformulations}

\subsection{Sharpness and the local Petrovskii condition}

Holomorphic sharpness of wavefronts can be detected by a topological criterion, the {\em local Petrovskii condition}, introduced in~\cite{ABG73} and reformulated in~\cite{Vassiliev86} in terms of the topology of Milnor fibers of generating functions. This condition generalizes a global homological condition from~\cite{Petrovskii45}.

This condition has sense for all real holomorphic function singularities with finite {\em Milnor numbers}. Let us remind the needed topological objects (see~\cite{AVG82,Milnor}). Let $f\colon \big({\mathbb C}^n,{\mathbb R}^n,0\big) \to ({\mathbb C},{\mathbb R},0)$ be a holomorphic function germ with $df(0)=0$; suppose that $0$ is an isolated critical point of~$f$, so that its Milnor number~$\mu(f)$ is finite. Let $B \subset {\mathbb C}^n$ be a sufficiently small ball centered at point~$0$, and $f_\lambda$ a very small perturbation of function $f$ which is non-discriminant (i.e., the variety $f_\lambda^{-1}(0)$ is non-singular in $B$). The manifold $V_\lambda \equiv f^{-1}_\lambda(0) \cap B$ is then called the {\em Milnor fiber} corresponding to the perturbation $f_\lambda$. It is a smooth $(2n-2)$-dimensional manifold with boundary $\partial V_\lambda \equiv V_\lambda \cap \partial B$, and is homotopy equivalent to the wedge of $\mu(f)$ spheres of dimension $n-1$. In particular $\tilde H_{n-1}(V_\lambda) \simeq {\mathbb Z}^{\mu(f)}$ and
 \begin{gather}\tilde H_{n-1}(V_\lambda, \partial V_\lambda) \simeq {\mathbb Z}^{\mu(f)}; \label{miln}
\end{gather}
here $\tilde H$ denotes the reduced homology groups, so that $\tilde H_{n-1} \equiv H_{n-1}$ if $n>1$.

If perturbation $f_\lambda$ is {\em real} (that is, $f_\lambda({\mathbb R}^n) \subset {\mathbb R}$), then there are two important elements of the group~(\ref{miln}): the {\em even} and the {\em odd local Petrovskii classes}.

The {\it even Petrovskii class} is realized by the fundamental class of the manifold $V_\lambda \cap {\mathbb R}^n$ of real points of the Milnor fiber, oriented as the boundary of the domain where $f_\lambda\leq 0$. For example, if the function $f$ has a local minimum at its critical point, then the even Petrovskii class of its perturbation $f+ \varepsilon$, $\varepsilon>0$, is trivial: it is represented by the empty cycle.

 In this work we will deal with the {\it odd Petrovskii class} whose construction is more tricky, see, e.g.,~\cite[Chapter~IV]{APLT}. However, we will be mainly interested not in this class itself but in the condition of its triviality, which has an easy characterization: the odd Petrovskii class associated with the perturbation $f_\lambda(x_1, \dots, x_n)$ of function $f$ is equal to zero if and only if the even Petrovskii class of perturbation $f_\lambda(x_1, \dots, x_n)+ x_{n+1}^2-x_{n+2}^2$ of function $f(x_1, \dots, x_n)+ x_{n+1}^2-x_{n+2}^2$ is equal to zero.

\begin{Definition} \rm
The {\it local Petrovskii class} in the group (\ref{miln}) is defined as the even local Petrovskii class if $n$ is even or the odd local Petrovskii class if $n$ is odd. A non-discriminant point $\lambda$ satisfies the {\it local Petrovskii condition} if the corresponding local Petrovskii class is equal to $0$.
\end{Definition}

\begin{Proposition}\label{sence}If $f$ is the generating function of the projectivized wavefront~$W^*(F)$ of a~strictly hyperbolic operator at its point~$X$, the corresponding critical point of~$f$ is isolated, and~$f_{\bf x}$ is a small non-discriminant real perturbation of~$f$ of the form~\eqref{prover}, then the principal fundamental solution~$E(F)$ is holomorphically sharp at~$X$ in the local $($close to $X)$ component of ${\mathbb R}P^{N-1}\setminus W^*(F)$ containing point~${\bf x}$ if and only if the corresponding local Petrovskii class is equal to~$0$ in group $\tilde H_{N-3}(V_{\bf x}, \partial V_{\bf x})$.
\end{Proposition}

Part ``if'' of this statement is proved in \cite{ABG73}, part ``only if'' in \cite{Vassiliev86}.

\subsection{Functoriality of local Petrovskii classes under adjacencies of singularities}

Let $\Phi\colon \big({\mathbb C}^n \times {\mathbb C}^l, {\mathbb R}^n \times {\mathbb R}^l, 0\big) \to ({\mathbb C}, {\mathbb R}, 0)$ be a deformation of our function germ $f$, parameterized by points $\lambda \in {\mathbb C}^l$, i.e., $\Phi$ is a family of functions $f_\lambda \equiv \Phi(\cdot,\lambda) \colon {\mathbb C}^n \to {\mathbb C}$ such that $f_0 \equiv f$ and $f_\lambda\big({\mathbb R}^n\big) \subset {\mathbb R}$ if $\lambda \in {\mathbb R}^l \subset {\mathbb C}^l$.

Let $\Sigma \subset {\mathbb R}^l$ be the {\em discriminant} of this deformation, i.e., the set of parameters $\lambda \in {\mathbb R}^l$ such that variety $V_\lambda$ is singular in~$B$. This set contains the {\em real discriminant} (see Section~\ref{prove}) but generally can be greater than it, including also parameter values $\lambda$ of functions $f_\lambda$ having imaginary critical points in~$B$ with zero critical values. This set divides parameter space ${\mathbb R}^l$ into several connected components in a neighbourhood of the origin.

By the construction of Petrovskii classes, the local Petrovskii condition is satisfied or is not satisfied simultaneously for all points from any such component.

Consider some such component~$C$; let $\tilde \lambda \in \Sigma \cap \bar C$ be a discriminant value of the parameter, which lies in the boundary of~$C$. By definition of the discriminant, the corresponding function~$f_{\tilde \lambda}$ has several (at least one) critical points $a_j\in B$ with critical value~$0$. Consider a set of very small non-intersecting balls $B_j \subset B$ around all these points.

Let $\overline{\lambda} \in C$ be a non-discriminant value of the parameter which is very close to $\tilde \lambda$, so that the corresponding level set $V_{\overline{\lambda}}$ is transversal to the boundaries of all balls~$B_j$, and varieties $V_{\overline{\lambda}} \cap B_j$ can be considered as the Milnor fibers of the corresponding critical points~$a_j$ of function $f_{\tilde \lambda}$.

\begin{Proposition}\label{funct}\quad
\begin{enumerate}\itemsep=0pt
\item[$1.$] For any real critical point $a_j \in B$ of function $f_{\tilde \lambda}$ with $f_{\tilde \lambda}(a_j)=0$, the odd $($respectively, even$)$ Petrovskii class of singularity $(f_{\tilde \lambda},a_j)$ in the homology group $\tilde H_{n-1}(V_{\overline{\lambda}} \cap B_j, V_{\overline{\lambda}} \cap \partial B_j)$ of the corresponding Milnor fiber is equal to the image of the odd $($respectively, even$)$ Petrovskii class of initial singularity $(f,0)$ under the obvious map
\begin{gather}\label{functor}
\tilde H_{n-1}(V_{\overline{\lambda}}, \partial V_{\overline{\lambda}}) \to
\tilde H_{n-1}\big(V_{\overline{\lambda}}, V_{\overline{\lambda}} \cap \overline{(B\setminus B_j)}\big) \equiv
\tilde H_{n-1}(V_{\overline{\lambda}} \cap B_j, V_{\overline{\lambda}} \cap \partial B_j).
\end{gather}
\item[$2.$] For any imaginary critical point $a_j \in B$ of $f_{\tilde \lambda}$ with $f_{\tilde \lambda}(a_j)=0$, the images of the odd and even local Petrovskii classes under the map~\eqref{functor} are equal to~$0$.
\item[$3.$] If deformation $\Phi$ has the form \eqref{prover}, then the corresponding principal fundamental solution~$E(F)$ is holomorphically sharp at point $\tilde \lambda$ of the wavefront in component $C$ if and only if the images of the local Petrovskii class of $f_{\overline{\lambda}}$ under all maps \eqref{functor} corresponding to all such critical points $a_j$ are equal to zero in all groups $\tilde H_{N-3}(V_{\overline{\lambda}} \cap B_j)$.
\end{enumerate}
\end{Proposition}

Statements~1 and~2 of this proposition follow immediately from the construction of Petrovskii classes, and Statement~3 from the multisingularity version of Proposition~\ref{sence} also proved in~\cite{Vassiliev86}.

If $f$ and $g$ are two germs of functions $\big({\mathbb C}^n,{\mathbb R}^n,0\big) \to ({\mathbb C}, {\mathbb R}, 0)$ with ${\rm d}f(0)=0={\rm d}g(0)$, which are $\operatorname{Diff}_0$-equivalent (i.e., can be transformed one to another by the composition with a local diffeomorphism $\big({\mathbb C}^n,{\mathbb R}^n,0\big) \to \big({\mathbb C}^n,{\mathbb R}^n,0\big)$), and $\Phi$ and $\Gamma$ are some their real versal deformations, then there is a one-to one correspondence between the local connected components of the complements of the discriminants of these deformations. This correspondence is defined by maps inducing deformations equivalent to~$\Phi$ and~$\Gamma$ one from the other in accordance with the definition of versal deformations (see~\cite{AVG82} and map~$\theta$ in~(\ref{versdef})). The Petrovskii conditions are respected by this correspondence. Therefore to prove Statement~2 of Theorem~\ref{maint} for a parabolic singularity class it is enough to present a component of the complement of the discriminant set of an arbitrary real versal deformation of an arbitrary singularity of this class, such that the local Petrovskii condition will be satisfied in this component at all the discriminant points of its boundary except for the most singular points corresponding to the parabolic singularity class itself. For instance, we can consider only singularities~$f(\xi)$ given by the polynomials from Table~\ref{t4}, and their monomial versal deformations of form
\begin{gather}\label{mondef}
\Phi(\xi,\lambda)\equiv f(\xi)+ \sum_{\alpha \in M} \lambda_\alpha \xi^\alpha ,
\end{gather}
where $M \subset {\mathbb Z}^n_+$ is a finite set of multiindices $\alpha=(\alpha_1, \dots, \alpha_n)$, $\xi^\alpha \equiv \xi_1^{\alpha_1} \cdots \xi_n^{\alpha_n}$.

If $\Phi(\xi_1, \dots, \xi_n; \lambda)$ is a versal deformation of function $f(\xi_1, \dots, \xi_n)$ of the form (\ref{mondef}), then
$\Phi(\xi_1, \dots, \xi_n; \lambda)+\big(\xi_{n+1}^2 + \xi_{n+2}^2\big)$ and $\Phi(\xi_1, \dots, \xi_n; \lambda)-\big(\xi_{n+1}^2 + \xi_{n+2}^2\big)$ are versal deformations respectively of functions $f(\xi_1, \dots, \xi_n)+\big(\xi_{n+1}^2 + \xi_{n+2}^2\big)$ and $f(\xi_1, \dots, \xi_n)-\big(\xi_{n+1}^2 + \xi_{n+2}^2\big)$ in $n+2$ variables, and have the same discriminant sets in the spaces of parameters $\lambda = \{\lambda_\alpha\}$. The homology groups (both absolute and modulo the boundary) of the middle dimensions of Milnor fibers of these three functions corresponding to the same non-discriminant parameter values~$\lambda$ are naturally isomorphic to one another. These isomorphisms preserve the Petrovskii condition (see \cite[Section~V.1.7]{APLT}). Therefore proving Theorem~\ref{maint} (in its Case~(a)) we can assume that for singularity classes $P_8$ we have $n=5$, $Q=\xi_4^2- \xi_5^2$, and for singularities $X_9$ or $J_{10}$ we have $n=3$, $Q=\xi_3^2$.

\begin{Proposition}\label{lem11}If function $f\colon \big({\mathbb C}^n,{\mathbb R}^n,0\big) \to ({\mathbb C}, {\mathbb R}, 0)$ is given by one of normal forms from Table~{\rm \ref{t4}}, the dimension $n$ is odd and the positive inertia index of quadratic form $Q$ also is odd, then the boundary of the odd Petrovskii class of any non-discriminant small perturbation~$f_\lambda$ of~$f$ is a non-trivial element of group $H_{n-2}(V_{\lambda} \cap \partial B)$.
\end{Proposition}

\begin{proof}[Scheme of the proof] A basis in $H_{n-1}(V_\lambda)$ can be composed of {\em vanishing cycles} $\Delta_i$, $i=1, \dots, \mu(f)$, see, e.g., \cite{AVG84,APLT}.
By the Poincar\'e duality and exact sequence of pair $(V_\lambda, \partial V_\lambda)$, the desired non-triviality of the boundary for a morsification $f_\lambda$ is equivalent to the following condition: vector $\Pi(\lambda)$ of intersection indices of the odd Petrovskii class with these basic vanishing cycles is not an integer linear combination of rows of intersection matrix $\{\langle \Delta_i,\Delta_j\rangle \}$ of these cycles.

It is enough to prove the latter property for an arbitrary single non-discriminant real morsification $f_\lambda$.
Indeed, by the exact formulas for the local Petrovskii classes (see \cite[Sections~V.1 and~V.4]{APLT}) expressing them in terms of intersection indices and the Morse indices of real critical points, vectors $\Pi(\lambda)$ for morsifications~$f_\lambda$ from the neighbouring components of the complement of the discriminant differ by adding or subtracting a row of the intersection matrix.

So, it remains to calculate all these data for one arbitrary morsification of any of our singularity classes. Vectors~$\Pi(\lambda)$ for these morsifications follow after that by the above-mentioned explicit formulas, and give the promised result.

The intersection indices of vanishing cycles for corank~2 singularities $X_9$ and $J_{10}$ can be found by the Gusein-Zade--A'Campo method, see~\cite{Gusein-Zade74} and~\cite{A'Campo75}, starting, e.g., from the morsifications discussed in Proposition~\ref{perts} below. For singularity class~$P_8^1$ (represented by function $\xi_1^3 + \xi_2^3+ \xi_3^3 + Q$) these indices follow from the main theorem of~\cite{Gab}, and for class $P_8^2$ they were calculated in \cite[Section~7]{V16}.
\end{proof}

\begin{Remark}\label{rem3} Statement 1 of Theorem~\ref{maint} follows immediately from this proposition and Proposition~\ref{sence}.
\end{Remark}

\section[Proof of Statement 2 of Theorem \ref{maint} for parabolic singularities of corank 2]{Proof of Statement 2 of Theorem \ref{maint}\\ for parabolic singularities of corank 2}\label{cor2}

\subsection{Examples}\label{theex}

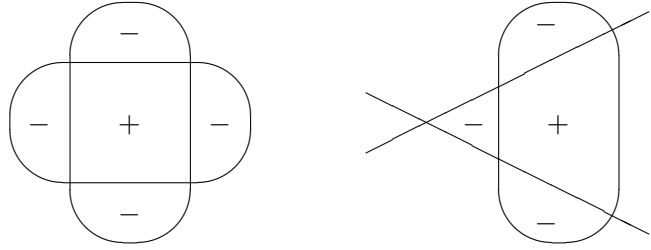
\begin{figure}[t]\centering\unitlength=0.8mm
\begin{picture}(40,42)
\put(20,20){\oval(20,40)}
\put(18,3.5){$-$}
\put(18,33.5){$-$}
\put(3,18.5){$-$}
\put(33,18.5){$-$}
\put(18,18.5){$+$}
\put(20,20){\oval(40,20)}
\end{picture}
\qquad
\begin{picture}(50,42)
\put(40,20){\oval(20,40)}
\put(8,15){\line(2,1){47}}
\put(8,25){\line(2,-1){47}}
\put(24,18.5){$-$}
\put(36,2){$-$}
\put(36,35){$-$}
\put(38,18.5){$+$}
\end{picture}

\caption{$+X_9$ (left), $X_9^1$ (right).}\label{X90}
\end{figure}

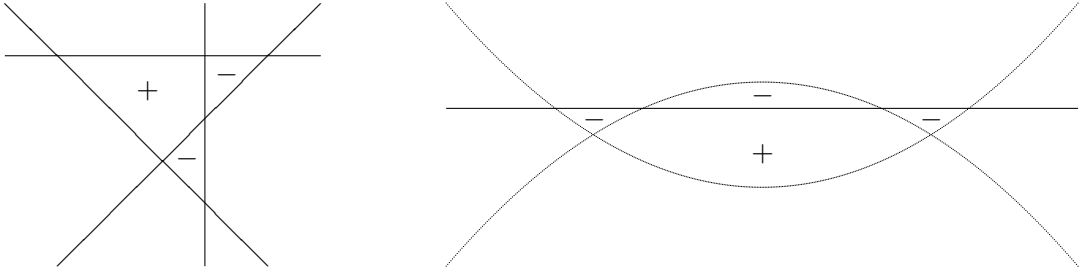
\begin{figure}\centering\unitlength=0.7mm
\begin{picture}(60,53)
\put(0,40){\line(1,0){60}}
\put(38,0){\line(0,1){50}}
\put(10,0){\line(1,1){50}}
\put(50,0){\line(-1,1){50}}
\put(40,35){$-$}
\put(25,32){$+$}
\put(32.5,19){$-$}
\end{picture} \qquad \qquad
\begin{picture}(120,53)
\put(0,30){\line(1,0){120}}
\bezier{400}(0,50)(60,-20)(120,50)
\bezier{400}(0,0)(60,70)(120,0)
\put(58,20){$+$}
\put(26,26.5){$-$}
\put(90,26.5){$-$}
\put(58,31){$-$}
\end{picture}
\caption{$X_9^2$ (left), $J_{10}^3$ (right).}\label{X92}
\end{figure}

\begin{figure}\centering\unitlength=0.8mm
\begin{picture}(120,57)
\bezier{300}(60,30)(30,40)(20,40)
\bezier{300}(60,30)(30,20)(20,20)
\bezier{200}(20,40)(0,40)(0,30)
\bezier{200}(20,20)(0,20)(0,30)
\bezier{300}(60,30)(90,40)(100,40)
\bezier{300}(60,30)(90,20)(100,20)
\bezier{200}(100,40)(120,40)(120,30)
\bezier{200}(100,20)(120,20)(120,30)
\bezier{300}(60,55)(80,0)(90,0)
\bezier{300}(120,55)(100,0)(90,0)
\put(20,29){$-$}
\put(65,29){$-$}
\put(114,29){$-$}
\put(90,29){$+$}
\put(90,12){$-$}
\end{picture}
\caption{$J_{10}^1$.} \label{J101}
\end{figure}

\begin{Proposition}\label{perts}
Any function $\varphi(\xi_1,\xi_2)$ in two variables defined by one of normal forms indicated in Table~{\rm \ref{t4}} for singularities of types $+X_9$, $X_9^1$, $X_9^2$, $J_{10}^3$, or $J_{10}^1$ has an arbitrarily small perturbation, whose set of zeros and domains of constant signs of values look topologically as shown in Figs.~{\rm \ref{X90}} $($left and right$)$, {\rm \ref{X92}} $($left and right$)$, or~{\rm \ref{J101}} respectively. $($The non-closed curves in these pictures are assumed to be continued to the infinity in~${\mathbb R}^2.)$

In particular, any versal deformation of $\varphi$ contains perturbations with this topological picture.
\end{Proposition}

 Proof is elementary.

For any of these five functions $\varphi(\xi_1, \xi_2)$ let us add a small positive constant to
its perturbation shown in the corresponding picture in such a way that all critical values at the minima points will remain negative. All the other critical values will become positive, therefore the obtained function is non-discriminant. Denote this function by $\tilde \varphi$ and define function $\tilde f\colon {\mathbb C}^3 \to {\mathbb C}$ by
\begin{gather}\label{tfi}
\tilde f(\xi_1, \xi_2, \xi_3) \equiv \tilde \varphi(\xi_1,\xi_2)+\xi^2_3 ;
\end{gather}
this is a perturbation of function
\begin{gather}\label{fi}
f \equiv \varphi(\xi_1,\xi_2)+\xi^2_3 .
\end{gather}

\begin{Proposition}\label{imag}
If $f$ is the function \eqref{fi}, where $\varphi$ is one of our five polynomials from Table~{\rm \ref{t4}}, and~$\tilde f$ is the perturbation of~$f$ just defined, then any Morse perturbation $f_\lambda$ of $f$ lying in the same component of the complement of the discriminant variety of a deformation of~$f$ as~$\tilde f$ can have no more than one pair of imaginary critical points $($so that the number of its real critical points is equal to either~$\mu(f)$ or~$\mu(f)-2)$.
\end{Proposition}

\begin{proof} The sum of local indices of vector field $\operatorname{grad} f_\lambda$ at all real critical points of $f_\lambda$ with negative (respectively, positive) critical values is an invariant of components of the complement of the discriminant. For our five perturbations these sums are equal to: 4 and $-3$ for $+X_9$, 3 and $-4$ for $X_9^1$, 2 and $-5$ for $X_9^2$, 3 and $-5$ for $J_{10}^3$, 4 and $-4$ for $J_{10}^1$. The index of any Morse critical point is equal to $1$ or $-1$, hence $f_\lambda$ should have at least 7 real critical points in the case of any $X_9$ singularity and at least 8 in the case of $J_{10}$.
\end{proof}

\begin{Proposition}\label{mp}
Let $f$ be the function \eqref{fi} where $\varphi$ is given by one of formulas of Table~{\rm \ref{t4}} for one of five singularity classes $+X_9$, $X_9^1$, $X_9^2$, $J_{10}^3$ or $J_{10}^1$; let $C \subset {\mathbb R}^l$ be the component of the complement of discriminant variety $\Sigma$ in the parameter space of a deformation of $f$ containing the corresponding perturbation~\eqref{tfi}; let
$\tilde \lambda \in \Sigma$ be an arbitrary discriminant point in the boundary of this component $C$ which is sufficiently close to the origin in the parameter space of the deformation, and all critical points of $f_{\tilde \lambda}$ with zero critical value are not parabolic $($i.e., do not belong to the same singularity class as~$f)$. Then for any parameter value $\overline{\lambda} \in C$, which is sufficiently close to $\tilde \lambda$, and any critical point $a_j\in B$ of $f_{\tilde \lambda}$ with critical value~$0$ the corresponding localized odd Petrovskii class in $\tilde H_{2}(V_{\overline{\lambda}} \cap B_j , V_{\overline{\lambda}} \cap \partial B_j)$ is equal to~$0$.
\end{Proposition}

Statement 2 (Case~(a)) of Theorem~\ref{maint} for these five singularity classes follows immediately from this proposition and Proposition~\ref{funct}.

Proof of Proposition \ref{mp} takes the remaining part of this section. Everywhere in it we assume that $f_{\tilde \lambda}$ satisfies the conditions of this proposition.

We can and will assume that our deformation is versal, otherwise we can expand it to a versal one. Let $l$ be the dimension of its parameter space.

Let $\mu_j$, $j=1,\dots,$ be the Milnor numbers of all critical points $a_j$ of $f_{\tilde \lambda}$ with critical value 0. Suppose that for one of these points, say $a_1$, the localized Petrovskii class in $H_{2}(V_\lambda \cap B_1, V_\lambda \cap \partial B_1)$, where $\lambda \in C$ is very close to $\tilde \lambda$, is non-trivial. In Section~\ref{condit} we will show that in this case component~$C$ should contain functions $f_{\overline{\lambda}}$ satisfying at least one of two conditions described there; in Sections~\ref{propr} and~\ref{modif} we present a combinatorial program proving that such functions actually do not exist.

\subsection{Conditions} \label{condit}
By Statement~2 of Proposition~\ref{funct}, critical point~$a_1$ mentioned in the previous paragraph is real.

\begin{Lemma}\label{lem8}
We can assume that $f_{\tilde \lambda}$ has only one this critical point $a_1$ with critical value~$0$ $($i.e., if there are several such critical points of~$f_{\tilde \lambda}$, then there is another point $\tilde \lambda'$, $\tilde \lambda' \approx \tilde \lambda,$ in the boundary of~$C$ with unique such critical point and non-trivial localization of the Petrovskii class at neighboring points of~$C)$.
\end{Lemma}

\begin{proof} \looseness=-1 Suppose first that all critical points of $f_{\tilde \lambda}$ with critical value 0 are real. If $\tilde \lambda$ is sufficiently close to the origin in parameter space ${\mathbb R}^l$ of our versal deformation of function $f$, then a small open ball $U$ in this space with the center at point $\tilde \lambda$ is also the parameter space of a versal deformation of the multisingularity of function $f_{\tilde \lambda}$, formed by all these function germs $(f_{\tilde \lambda}, a_j)$. A miniversal deformation of such a multisingularity splits into the product of independent versal deformations of all singularities composing it, in particular its number of parameters is equal to $\sum \mu_j$. Therefore
our component $C \cap U$ of the complement of the discriminant set of this versal deformation of multisingularity $f_{\tilde \lambda}$ is ambient diffeomorphic to the direct product of some components $C_j$ of the complements of discriminant sets of all these singularities $(f_{\tilde \lambda},a_j)$, and additionally of space ${\mathbb R}^{l - \sum \mu_j}$. This diffeomorphism is realized by the restriction to $C \cap U$ of a~map inducing a deformation equivalent to our versal deformation of $f_{\tilde \lambda}$ from this miniversal one. (More precisely, this map is the projection map of a trivial fiber bundle, all whose fibers are diffeomorphic to ${\mathbb R}^{l - \sum \mu_j}$ and consist of functions, obtained one from another by diffeomorphisms of the argument space.)
Point $\tilde \lambda$ is sent by this diffeomorphism to the product of origins in the space ${\mathbb R}^{l - \sum \mu_j}$ and in all parameter spaces ${\mathbb R}^{\mu_j}$ of these deformations. Let us shift for any $j \neq 1$ the corresponding origin point of ${\mathbb R}^{\mu_j}$ to a non-discriminant point inside the corresponding component $C_j$. The obtained point of the product corresponds via our diffeomorphism to a point~$\tilde \lambda'$ in the boundary of our component $C \cap U$, such that $f_{\tilde \lambda'}$ has a single critical point with critical value~0, with
the same simple singularity class as $(f_{\tilde \lambda},a_1)$, and the localized Petrovskii class for a neighboring point $\overline{\lambda'} \in C$ again will be not equal to zero in group $\tilde H_{n-1}(V_{\overline{\lambda'}} \cap B_1 , V_{\overline{\lambda'}} \cap \partial B_1)$.

Suppose now that our function $f_{\tilde \lambda}$, $\tilde \lambda \in \partial C$, has imaginary critical points with value 0. By Proposition \ref{imag} there can be only one pair of such points, and these two critical points should be Morse. In this case we consider a very similar splitting of a neighbourhood of point $\tilde \lambda$, in which one factor ${\mathbb R}^2$ is not the parameter space of a real singularity of $f_{\tilde \lambda}$, but the real part of the product of parameter spaces of both its complex conjugate critical points; the rest of the consideration remains as previously.
\end{proof}

So, we can and will assume that $f_{\tilde \lambda}$ has only one critical point $a_1 \in B$ with critical value~0. By Proposition~\ref{prim} this critical point~$a_1$ is simple. A neighbourhood of point $\tilde \lambda$ in ${\mathbb R}^l$ can be considered as the base of a versal deformation of critical point~$a_1$ of function $f_{\tilde \lambda}$. Then by Theorem~\ref{v92} we may assume that either critical point~$a_1$ is Morse or it is of type~$A_2$ and our component~$C$ is the ``bigger'' one with respect to the corresponding cuspidal edge.

If this point is Morse, then component~$C$ contains a point $\overline{\lambda}$ such that function $f_{\overline{\lambda}}$ is Morse and satisfies the following

\begin{Condition}\label{cond1} The intersection index of the Petrovskii class in $\tilde H_{2}(V_{\overline{\lambda}}, \partial V_{\overline{\lambda}})$ with a vanishing cycle in $\tilde H_{2}(V_{\overline{\lambda}})$ defined by the segment in ${\mathbb R}^1 \subset {\mathbb C}^1$ connecting the non-critical value~$0$ with either the smallest positive or the largest negative critical value of $f_{\overline{\lambda}}$ is not equal to~$0$.
\end{Condition}

Indeed, this is true for all values $\overline{\lambda} \in C$ which are sufficiently close to $\tilde \lambda$ and correspond to Morse functions.

Accordingly to the explicit formulas for the Petrovskii classes (see \cite{Vassiliev86, APLT}) this condition is equivalent to the following one.

\begin{ConditionN}
 Either the positive inertia index of the critical point of $f_{\overline{\lambda}}$ with the smallest positive critical value is odd, or the positive inertia index of the critical point of $f_{\overline{\lambda}}$ with the largest negative critical value is even.
\end{ConditionN}

In the second case of $A_2$ singularity (assuming that Condition~\ref{cond1} is {\em not} satisfied for all points of $C$) our component $C$ contains points $\overline{\lambda}$ arbitrarily close to $\tilde \lambda$, such that $f_{\overline{\lambda}}$ is Morse and

\begin{Condition}\label{cond2}\quad
\begin{enumerate}\itemsep=0pt
\item[$1)$] $f_{\overline{\lambda}}$ has at least two real critical points with positive critical values;
\item[$2)$] positive inertia index of the quadratic part of the critical point of $f_{\overline{\lambda}}$ with the smallest positive critical value is even;
\item[$3)$] positive inertia index of the quadratic part of the critical point of $f_{\overline{\lambda}}$ with the next smallest positive critical value is exactly by~$1$ smaller than that for the previous one;

\item[$4)$] intersection index in $H_{2}(V_{\overline{\lambda}})$ of two vanishing cycles defined by simplest paths
\mbox{\begin{picture}(25,5)
\put(1,1){\circle{2}}
\put(-1.5,-6.5){\footnotesize $0$}
%\put(0,-1){\small $0$}
\put(13,1){\circle*{2}}
\put(25,1){\circle*{2}}
\put(12,1){\vector(-1,0){9.3}}
\bezier{120}(24,1.2)(13,6)(2.2,1)
\end{picture}}
connecting these critical values with $0$ in ${\mathbb C}^1$ is equal to $1$ or $-1$;

\item[$5)$] the intersection index of the odd Petrovskii class in $H_2(V_{\overline{\lambda}},\partial V_{\overline{\lambda}})$ with the vanishing cycle corresponding to the smallest positive critical value $($respectively, to the next smallest one$)$ is equal to~$0$ $($respectively, to~$1$ or $-1)$.
\end{enumerate}
\end{Condition}

It remains to prove that our component $C$ does not contain points~$\overline{\lambda}$ satisfying either of these two sets of conditions. To do it we apply the program~\cite{pro2} counting all topological types which the non-dis\-cri\-mi\-nant morsifications of function~$f$ can have.

\subsection{Description of the program} \label{propr}
For an extended description of this program see \cite[Section~V.8]{APLT} (although the version of the program quoted there is now obsolete, for the actual version see~\cite{pro2}). Let us remind its basic ideas.

Any generic real morsification $f_\lambda$ of a function singularity $f$ of corank $\leq 2$ in $n$ variables is characterized by a set of its topological invariants, including
\begin{enumerate}\itemsep=0pt
\item[a)] intersection indices of appropriately oriented and ordered vanishing cycles in $H_{n-1}(V_\lambda)$ corresponding to all its critical points (including the imaginary ones),
\item[b)] intersection indices of both local Petrovskii classes with all these vanishing cycles,
\item[c)] Morse indices of all real critical points (more precisely, the {\em positive} inertia indices of their quadratic parts) ordered by increase of their critical values,
\item[d)] the number of negative critical values, and
\item[e)] the number $n$ of variables (in fact only its residue modulo 4 is important).
\end{enumerate}

Any possible set of such data is called a {\em virtual morsification}.

If we know these data for an actual morsification~$f_\lambda$, then we can calculate them for all morsifications which can be obtained from~$f_\lambda$ by the standard surgeries, such as Morse surgeries of death/birth of real critical points of neighbouring indices, the change of the order in~${\mathbb R}^1$ of critical values at distant critical points, the passage of values at two distant imaginary critical points through the real axis, the jump of a critical value through~0 (this is the unique surgery changing the component of the complement of the discriminant), and the change of the choice of paths defining the cycles vanishing in imaginary critical points. Our algorithm starts from the virtual morsification corresponding to the initial perturbation $f_{\lambda} \in C$ and applies to it all possible chains of transformations of virtual morsifications modelling the possible surgeries of actual morsifications. The newborn virtual morsifications are compared with all found previously ones, and are added to the list if they are new. A priori this algorithm can provide virtual morsifications not corresponding to any actual one; however it surely finds all virtual morsifications corresponding to all actual ones.

\subsection{Modification of program \cite{pro2} for our purposes}\label{modif}

To apply this program to any of our five singularities, we calculate (by hands) data (a)--(d) for real morsification $f_\lambda$ from our component $C$ described in Section~\ref{theex}, put these data to our program, switch out in it the surgery of jumping of critical values through~$0$ (since we are interested only in the morsifications from component~$C$) and additionally include alarm operators aborting the program and typing the corresponding message if a newborn virtual morsification satisfies one of two conditions listed in Section~\ref{condit}.

More precisely, we take program \cite{pro2} and do the following:
\begin{enumerate}\itemsep=0pt
\item[a)] make in lines 1--7 changes described in lines 12--18 (the Milnor number of all $X_9$ singularities is equal to 9, and that of $J_{10}$ is equal to 10),
\item[b)] put the command ${\rm NPOZC}=4$ in line 37 for singularities $+X_9$ and $J_{10}^1$, ${\rm NPOZC}=3$ for~$X_9^1$ and~$J_{10}^3$, and ${\rm NPOZC}=2$ for $X_9^2$: this is the number of critical points with negative critical values for our starting morsification specifying component~$C$,
\item[c)] uncomment operators ${\rm L(MD+1)}=1$ and ${\rm L(MDD)}=1$ in lines~115,~116 (i.e., we disable the surgeries of crossing the discriminant, which correspond to the jumps of real critical values through~0),
\item[d)] insert the next text immediately after operator 2318 CONTINUE:
\begin{gather}
\begin{split}
 &\mbox{IF(PC1(NPOZC).NE.0)\ GOTO \ 7343}\\
 &\mbox{NPO1=NPOZC+1}\\
 & \mbox{IF(PC1(NPO1).NE.0)\ GOTO \ 7343}\\
 & \mbox{NPO2=NPOZC+2}\\
 & \mbox{IF(PC1(NPO2).NE.0.AND.INDC(NPO2).LT.INDC(NPO1)) \ GOTO \ 7343}
 \end{split}\label{pro3}
\end{gather}
(the first two operators ``IF'' here check Condition~\ref{cond1} above, and the third one detects items~3 and (partly)~5 of Condition~\ref{cond2}; on the address 7343 the program types the word ALARM and data of the virtual morsification for which one of these conditions was detected, and then terminates its work);
\item[e)] calculate the intersection indices of vanishing cycles of morsification~$f_\lambda$ described in Section~\ref{theex} using the Gusein-Zade--A'Campo method, and substitute all obtained {\em non-zero} indices $\langle\Delta_i,\Delta_j\rangle$, $i<j$, into the subroutine DATA at the very end of the program: e.g., if $\langle\Delta_3,\Delta_5\rangle=1$ then we insert operator $C(3,5)=1$ there. Also, if the positive inertia index of the quadratic part of the~$i$th (in the order of increase of critical values) critical point of our morsification is equal to some number $q \neq 2$, then we put ${\rm INDC}(i)=q$ in the last part of this subroutine. These indices for the critical points of functions~(\ref{tfi}) for our corank~2 singularities shown in Figs.~\ref{X90}--\ref{J101} are equal to~3 at minima of $\tilde \varphi$, 1 at maxima, and~2 at saddlepoints.
\end{enumerate}

All programs modified in this way and corresponding to our five corank~2 parabolic singularities can be found by the address
\url{https://drive.google.com/drive/folders/1L5p_HOvrBbcyBv-o4Ov4S_sZBbJytdxr?usp=sharing}.

The result of their work is negative: in all five cases the program detects that component $C$ does not contain points satisfying either of Conditions~\ref{cond1} or~\ref{cond2}.
This proves Proposition~\ref{mp}, and hence also Statement~2 of Theorem~\ref{maint} for our five singularities of corank~2.

\section[Singularities $P_8$]{Singularities $\boldsymbol{P_8}$}\label{p8}

In the case of singularities of corank $>2$ we need to apply a different program~\cite{pro}.
Indeed, we generally cannot predict the Morse indices of the newborn real critical points in the Morse birth surgery, if we know only the data of the virtual morsification before the birth: we can predict only parities of these indices. In the case of functions of corank 2, when the choice is between the pairs of neighboring Morse indices~$(0,1)$ or~$(1,2)$ of the new critical points of essential part~$\varphi_\lambda$ of the morsification, this is enough to obtain the exact values of the indices, but for the functions of greater coranks we need to use a program in which the Morse indices take only the values odd/even.

Fortunately, this is enough for our problem concerning the simplest singularities of corank 3. Other changes and features in these cases are as follows.

\subsection[Singularity class $P_8^1$]{Singularity class $\boldsymbol{P_8^1}$}

This class is represented by function $f \equiv \varphi(\xi_1, \xi_2, \xi_3) + \xi_4^2- \xi_5^2$, where $\varphi=\xi_1^3 + \xi_2^3 + \xi_3^3$. We consider the standard monomial versal deformation of this function depending on parameter $\lambda=(\lambda_1, \dots, \lambda_8)$,
\begin{gather}\label{verp8}
f_\lambda(\xi) \equiv f(\xi) + \lambda_1 + \lambda_2\xi_1+\lambda_3\xi_2+\lambda_4\xi_3+\lambda_5\xi_1\xi_2 +\lambda_6 \xi_1\xi_3 + \lambda_7 \xi_2\xi_3 +\lambda_8 \xi_1\xi_2\xi_3 .
\end{gather}

The stratum of points of type $P_8$ is represented in it by the $\{\lambda_8\}$ axis. The multiplicative group of positive numbers acts on parameter space~${\mathbb R}^8$ of this deformation: any number $t$ sends a function $f_\lambda(x)$ to $t^{3}f_\lambda(x/t)$. In other words, $t(\lambda_1, \lambda_2, \lambda_3, \lambda_4, \lambda_5, \lambda_6, \lambda_7, \lambda_8)=\big(t^3\lambda_1, t^2\lambda_2, t^2\lambda_3, t^2 \lambda_4, t\lambda_5, t\lambda_6, t\lambda_7, \lambda_8\big)$. This action preserves the discriminant and allows us not to take care of the smallness of the perturbations except for parameter $\lambda_8$.

Let us construct a perturbation of function $f$, contained in component $C$ of the complement of the discriminant of this deformation, which will satisfy Statement~2 of Theorem~\ref{maint}.

Function $\varphi$ has a convenient morsification $\tilde \varphi \equiv \xi_1^3 + \xi_2^3 + \xi_3^3 - 3 (\xi_1 + \xi_2 +\xi_3)$ with eight real critical points: namely, it are all points with $\xi_i = \pm 1$. Let be $\tilde f \equiv \tilde \varphi + \xi_4^2-\xi_5^2$. The intersection indices of vanishing cycles in $H_{4}\big(\tilde f^{-1}(0)\big)$ related with these critical points can be calculated by the method of \cite{Gab}.

Consider family $\varphi_{(\tau)}$ of perturbations of $\varphi$ depending on parameter $\tau \in [0,1]$,
\begin{gather}\label{famil}
\varphi_{(\tau)} \equiv \varphi(\xi_1, \xi_2, \xi_3) + 6\tau(\xi_1\xi_2 +\xi_1\xi_3 + \xi_2\xi_3) - (3+12\tau)(\xi_1+\xi_2+\xi_3) .
\end{gather}
All functions of this family are invariant under the permutations of coordinates $\xi_1$, $\xi_2$, $\xi_3$, and have the fixed critical point~$(1,1,1)$. This family connects the morsification $\tilde \varphi \equiv \varphi_{(0)}$ with a~function having a singularity of class~$D_4^-$ at this point. The topological type of morsifications does not change along this path (except for the very last point), in particular the intersection indices of vanishing cycles and the Morse indices of critical points of the corresponding functions $f_{(\tau)} \equiv \varphi_{(\tau)} + \xi_4^2-\xi_5^2$ for nearly final values of $\tau \approx 1$ will be the same as for $\tau=0$.

A neighbourhood of the final function $\varphi_{(1)}$ in the parameter space of the versal deformation~(\ref{verp8}) can be considered as the parameter space of a versal deformation of the corresponding multisingularity consisting of this point of type $D_4^-$ and four distant Morse points, so we can deform these critical points independently.

Four Morse critical points of $\varphi_{(\tau)}$, $\tau =1-\varepsilon$, tending to the collision in this point of type~$D_4^-$, consist of three points with signature~$(2,1)$ and equal critical values, and one minimum point with a slightly smaller critical value tending to $-24$ when $\tau$ tends to~$1$. We can perform two standard Morse surgeries over them, first colliding this minimum point with one point of signature $(2,1)$, and then returning these critical points to the real domain as two points with signatures $(2,1)$ and $(1,2)$ and critical values slightly greater than these at two points of signature $(2,1)$ not participating in these surgeries. Composition of these two surgeries realizes the passage
{\small
\unitlength=0.8mm
$\mbox{\begin{picture}(52,13)
\put(0,10){\line(1,0){16}}
\put(3,13){\line(1,-2){7}}
\put(13,13){\line(-1,-2){7}}
\put(6.8,6.2){{\footnotesize $-$}}
{\Large \put(21,5){$\leftrightarrow$}}
\put(32,1){\line(1,0){15}}
\put(36,13){\line(1,-2){7}}
\put(41,13){\line(-1,-2){7}}
\put(37,2.5){{\footnotesize $+$}}
\end{picture}}$}
between the perturbations of a $D_4^-$-singularity. Other four critical points of $\varphi_{(1-\varepsilon)}$ can be perturbed very little during this surgery.

The resulting function has exactly three critical points of signature~$(2,1)$, which have the smallest critical values among all critical points of this function. Let us add a constant to this function in such a way that these three critical values of the obtained function $\overline{\varphi}$ will remain negative, and all the other ones become or remain positive. We claim that the component of the complement of the discriminant of the deformation~(\ref{famil}) containing function $\overline{f}\equiv \overline{\varphi} + \xi_4^2-\xi_5^2$ satisfies an analog of Proposition~\ref{mp}: the local Petrovskii condition holds in this component at all discriminant points of its boundary except for the most singular points of type~$P_8^1$.

The proof of this claim almost repeats that of Proposition~\ref{mp}.

In particular, the exact analogs of Proposition~\ref{imag} and Lemma~\ref{lem8} hold in this case together with their proofs: sums of indices of the vector field $\operatorname{grad} \overline{\varphi}$ over the critical points with negative (respectively, positive) critical values are equal to~$-3$ and~$3$.

The topological data of the obtained morsification~$\overline{f}$ can be easily derived from these for morsification $\varphi_{(1-\varepsilon)}$ (e.g., with the help of our program~\cite{pro}). Then
 we substitute these data to a modification of our program, see the concluding subroutine DATA in program mP81 from the folder quoted in Section~\ref{modif}. This program is obtained from program~\cite{pro} by almost the same changes by which we obtained in Section~\ref{modif} similar programs for corank~2 singularities from program~\cite{pro2}, with only the following exceptions.
The operator detecting cuspidal edge (see the fifth line in~(\ref{pro3})) should be slightly changed, because we generally do not know the integer values of Morse indices of critical points. Instead we write the command
\[\mbox{IF(PC1(NPO2).NE.0.AND.C(NPO1,NPO2).EQ.1) GOTO 7343}\]
checking items~4 and~5 of Condition~\ref{cond2} (we use a definite choice of orientations of vanishing cycles corresponding to real critical points, which allows us to fix the sign of their intersection index). Also, we write ${\rm INDC}(q)=-1$ in the concluding subprogram DATA for all numbers~$q$ of critical values of the function $\overline{\varphi}$ corresponding to critical points with {\em odd} positive inertia indices of quadratic parts; we write ${\rm NPOZC}=3$ in line~37 and activate operators ${\rm L}({\rm MD}+1)=1$ and ${\rm L(MDD)}=1$ in lines~115 and~116. At the beginning of this subprogram we preserve operator $\mbox{N}=-1$ (which means that number $n= 5$ of variables is as odd as the number~3 considered previously) but then write $\mbox{N2}=1$ (which means that number $n(n-1)/2$ for $n=5$ is even, unlike the case of $n=3$).

The obtained program proves that our component $C$ of the complement of the discriminant does not contain any morsifications satisfying Conditions~\ref{cond1} or~\ref{cond2} from Section~\ref{condit}, thus proving Statement~2 of Theorem~\ref{maint} for singularity class~$P_8^1$.

\subsection[Singularity $P_8^2$]{Singularity $\boldsymbol{P_8^2}$}

In this case we take initial morsification $\varphi_{\tilde \lambda} (\xi_1, \xi_2, \xi_2) + \xi_4^2-\xi_5^2$, where function $\varphi_{\tilde \lambda}$ in three variables is described in Section~7 of~\cite{V16} (and is called~$f_2$ there). This function $\varphi_{\tilde \lambda}$ has four critical points with signature of the quadratic part equal to $(2,1)$ and critical values equal to~0,~0,~0, and~1. In addition it has four critical points with signature $(1,2)$; the values at three of them are slightly greater than~$0$, and the value at the fourth one is also equal to~1. Let us subtract a very small constant from this function, making critical values at the first three critical points negative, and leaving the remaining five values positive. According to~\cite{V16}, the intersection matrix of corresponding vanishing cycles is then expressed by formula~(5) from~\cite{V16} with $X=0$, $Y=1$, $Z=0$, $W=-2$. Let us substitute these topological data into the same version of program~\cite{pro} which was used for $P_8^1$, again write $\mbox{NPOZC}=3$ in line~37, and run the obtained program~mP82.

It gives us a disappointing answer: one can see the word ALARM in the print-out of this program, followed by the data of a virtual morsification, for which Condition~\ref{cond1}$'$ is satisfied: namely, the critical point with the smallest positive critical value has an odd positive inertia index of the quadratic part.

We will show now that this is not dangerous for us. Indeed, there are the following two possible interpretations of this calculation.

1. All chains of admissible virtual surgeries leading to the virtual morsifications satisfying Condition~\ref{cond1}$'$, are fake (i.e., cannot be realized by chains of real surgeries arising along any generic paths in the parameter space of a~versal deformation). In this case the component of the complement of the discriminant, which contains our morsifications, is the desired one. Indeed, let us modify our program~mP82 in such a way that it checks only Condition~\ref{cond2} but forgives Condition~\ref{cond1}$'$ (i.e., we disable the first two operators IF in~(\ref{pro3})). The obtained program assures us that the morsifications satisfying this Condition~\ref{cond2} indeed do not occur in this component (while the morsifications satisfying Condition~\ref{cond1}$'$ do not occur by our conjecture).

2. If the conjecture of the previous paragraph is wrong, and our component contains real morsifications satisfying Condition~\ref{cond1}$'$, then there is a neighbouring component satisfying the assertion analogous to Proposition~\ref{mp} (i.e., such that the localized Petrovskii condition is satisfied in this component at all discriminant points of its boundary except for the most singular points of class~$P_8^2$). Indeed, any morsification of~$f$ from our component has at least three critical points with negative critical values and odd positive inertia indices, and also at least three points with positive critical values and even negative inertia indices (see the proof of Proposition~\ref{imag}). Moreover, quantities of all critical points with negative (or positive) critical values are odd (and hence equal to~3 or~5). Therefore if we have a real morsification in this component, satisfying Condition~\ref{cond1}$'$, then we can add to it a constant function in such a way that the resulting morsification has exactly four critical points with negative values and odd positive inertia indices, and four critical points with positive critical values and even positive inertia indices.

Any component of the complement of the discriminant, which contains such a morsification, is a desired one. Indeed, no Morse surgery can be performed over functions inside this component, hence {\em all} morsifications from it have the same signatures, and Conditions~\ref{cond1}$'$ or~\ref{cond2} cannot be satisfied for them.

\subsection*{Acknowledgements}

This work was supported by Russian Science Foundation grant 16-11-10316.

\pdfbookmark[1]{References}{ref}
\LastPageEnding

\end{document}